\documentclass[12pt,twoside]{amsart}
\usepackage{amssymb}
\usepackage{verbatim}
\usepackage{amsmath}
\usepackage{bm}
\usepackage{a4wide}
\usepackage[latin1]{inputenc}
\usepackage[T1]{fontenc}
\usepackage{times}
\usepackage{amssymb,latexsym}
\usepackage{enumerate}
\usepackage[stable]{footmisc}
\usepackage{color}
\usepackage[colorlinks,linkcolor=black,citecolor=black,urlcolor=black]{hyperref}

\makeatletter
\newcommand{\sumprime}{\if@display\sideset{}{'}\sum%
            \else\sum'\fi}
\makeatother

\begin{document}

\numberwithin{equation}{section}

\newtheorem{theorem}{Theorem}[section]
\newtheorem{proposition}[theorem]{Proposition}
\newtheorem{conjecture}[theorem]{Conjecture}
\def\theconjecture{\unskip}
\newtheorem{corollary}[theorem]{Corollary}
\newtheorem{lemma}[theorem]{Lemma}
\newtheorem{observation}[theorem]{Observation}
\newtheorem{definition}{Definition}
\numberwithin{definition}{section} 
\newtheorem{remark}{Remark}
\def\theremark{\unskip}
\newtheorem{kl}{Key Lemma}
\def\thekl{\unskip}
\newtheorem{question}{Question}
\def\theexample{\unskip}
\newtheorem{problem}{Problem}
\newtheorem{example}{Example}

\thanks{Supported by National Natural Science Foundation of China, No. 12271101.}

\address [Bo-Yong Chen] {School of Mathematical Sciences, Fudan University, Shanghai, 200433, China}
\email{boychen@fudan.edu.cn}

\address [Yuanpu Xiong] {School of Mathematical Sciences, Fudan University, Shanghai, 200433, China}
\email{ypxiong@fudan.edu.cn}

\title{Type problem and the first eigenvalue}
\author{Bo-Yong Chen and Yuanpu Xiong}

\date{}

\begin{abstract}
In this paper,  we study the relationship between the type problem and the asymptotic behavior of the first eigenvalues $\lambda_1(B_r)$ of ``balls'' $B_r:=\{\rho<r\}$ on a complete Riemannian manfold $M$ as $r\rightarrow +\infty$, where $\rho$ is a Lipschitz continuous exhaustion function with $|\nabla\rho|\leq1$ a.e. on $M$.  We show that $M$  is hyperbolic whenever 
\[
\Lambda_*:= \liminf_{r\rightarrow +\infty} \{ r^2 \lambda_1(B_r)\} >18.624\cdots.
\]
Moreover,  an upper bound of $\Lambda_*$ in terms of volume growth $\nu_*:=\liminf_{r\rightarrow +\infty} \frac{\log |B_r|}{\log r}$ is given as follows
\[
{\Lambda_*} \lesssim \begin{cases}
\nu_*^2,\ \ \ &\nu_*\gg1,\\
\nu_*\log\frac{1}{\nu_*},&1<\nu_*\ll1.
\end{cases}
\]
The exponent $2$ for $\nu_*\gg1$ turns out to be the best possible.

\end{abstract}

\maketitle

\tableofcontents

\section{Introduction}

Let $(M,g)$ be a complete,  non-compact Riemannian manifold,  and denote by $\Delta$ the Laplace operator associated to $g$. An upper semicontinuous function $u$ on $M$ is called {\it subharmonic} if $\Delta u\ge 0$ holds in the sense of distributions.  If every negative subharmonic function on $M$ has to be a constant,  then $M$ is said to be {\it parabolic};  otherwise $M$ is called {\it hyperbolic}.  It is well-known that $M$ is parabolic (resp.  hyperbolic) if and only if the Green function $G_M(x,y)$ is infinite (resp.  finite) for all $x\neq{y}$;  or the Brownian motion on $M$  is recurrent (resp.  transient).

The type problem is how to decide the parabolicity and   hyperbolicity through intrinsic geometric conditions.  The case of surfaces is classical,  for  the type of $M$  depends only on the conformal class of $g$,  i.e.,  the complex structure determined by $g$.  Ahlfors \cite{Ahlfors35} and Nevanlinna \cite{Nevanlinna40} first showed that $M$ is parabolic whenever
\begin{equation}\label{eq:parabolic_volume_2}
\int^{+\infty}_1\frac{dr}{|\partial{B(x_0,r)}|}=+\infty,
\end{equation}
where $B(x_0,r)$ is the geodesic ball with center  $x_0\in{M}$ and radius $r$. 
The same conclusion was extended to  high dimensional cases by  Lyons-Sullivan \cite{LyonsSullivan} and Grigor'yan \cite{Grigoryan83,Grigoryan85}.   Moreover,   \eqref{eq:parabolic_volume_1} can be relaxed to
\begin{equation}\label{eq:parabolic_volume_1}
\int^{+\infty}_1 \frac{rdr}{|B(x_0,r)|}=+\infty
\end{equation}
 (cf. Karp \cite{Karp82}, Varopolous \cite{Varopoulos} and Grigor'yan \cite{Grigoryan83,Grigoryan85}, see also Cheng-Yau \cite{ChengYau75}).  We refer to the excellent survey \cite{Grigoryan} of Grigor'yan for other sufficient conditions of parabolicity.  
 
On the other side,  it seems more difficult to find sufficient conditions for hyperbolicity.  Yet there is a classical result  stating that $M$ is hyperbolic whenever  the first (Dirichlet) eigenvalue $\lambda_1(M)$ of $M$ is positive. Recall that 
$$
\lambda_1(M): = \lim_{j\rightarrow +\infty} \lambda_1(\Omega_j)
$$
for some/any increasing sequence of precompact open sets $\{\Omega_j\}$ in $M$,  such that $M=\bigcup \Omega_j$.  Here given a precompact open set $\Omega\subset M$,   define 
$$
\lambda_1(\Omega):=\sup\left\{\frac{\int_\Omega |\nabla \phi|^2 dV}{\int_\Omega \phi^2 dV}:\phi\in \mathrm{Lip}_{\mathrm{loc}}(M),\ \mathrm{supp}\,\phi\subset\overline{\Omega},\ \phi\not\equiv 0 \right\},
$$
Sometimes, it is also natural to consider the bottom $\lambda_1^{ess}(M)$ of the {\it essential} spectrum instead of $\lambda_1(M)$, in connection with the geometry at infinity. Recall that
$
\lambda_1^{ess}(M): = \lim_K \lambda_1(M\setminus K)
$
with $K$ running through all compact subsets of $M$. Clearly, $\lambda_1^{ess}(M)\geq\lambda_1(M)$. The following result which is probably known, but we are unable to find it in literature.

\begin{theorem}\label{th:ess}
 $M$ is hyperbolic if $M$ has infinite volume and $\lambda_1^{ess}(M)>0$.  In other words, if $M$ is parabolic, then either $M$ has finite volume or $\lambda_1^{ess}(M)=0$. 
\end{theorem}

As an interesting consequence of Theorem \ref{th:ess},  we shall present a criterion for conformal finiteness of parabolic Riemann surfaces.  Recall that a Riemann surface is said to be {\it conformally finite} if it is conformally equivalent to a compact Riemann surface with finite punctures.  
\begin{corollary}\label{cor:ess}
Let $M$ be a parabolic Riemann surface which admits the (Poincar\'{e}) hyperbolic metric $g_{\rm hyp}$,  i.e.,   the universal covering of $M$ is the unit disc.   Then $M$ is conformally finite if and only if there exist a Riemann surface $(\widetilde{M},g)$ and compacts $K\subset M$ and $\widetilde{K}\subset \widetilde{M}$ such that  $(M\setminus K,g_{\mathrm{hyp}})$ is quasi-isometric to  $(\widetilde{M}\setminus \widetilde{K},g)$,  where  $g$ is  $d$-bounded in the sense of Gromov \cite{Gromov91},  that is,   the K\"ahler form of $g$ may be written as $d\theta$ for some smooth $1$-form $\theta$ on $\widetilde{M}\setminus\widetilde{K}$ such that the length $|\theta|_g$ of $\theta$\/ is uniformly bounded.
\end{corollary}

Recall that two Riemannian manifolds $(M_1,g_1)$ and $(M_2,g_2)$ are {\it quasi-isometric} if there exists a  quasi-isometry  $F: M_1\rightarrow M_2$,  that is,  $F$ is a diffeomorphism from $M_1$ onto $M_2$ such that
 for suitable  constant $C\ge  1$,
$$
C^{-1}{\rm dist}_{M_1}(x,y) \le {\rm dist}_{M_2}(F(x),F(y))\le C\, {\rm dist}_{M_1}(x,y),\ \ \ \forall\,x,y\in M_1.
$$

\begin{remark}
Every parabolic Riemann surface admits a K\"ahler metric which is $d-$bounded outside a compact subset (see \cite{CF},  pp.  393--394). 
\end{remark}

The main focus of this paper is to determine the hyperbolicity in the case  $\lambda_1(M)=0$. Grigor'yan showed that $M$ is hyperbolic if the following Faber-Krahn type inequality holds:
$$
\lambda_1(\Omega) \ge f(|\Omega|),\ \ \ \forall\,\Omega\subset\subset M: |\Omega|\ge v_0>0,
$$
where $f$ is a positive decreasing function on $(0,+\infty)$ such that $\int^{+\infty}_{v_0} \frac{dv}{v^2 f(v)}<+\infty$ (see, e.g., \cite{Grigoryan}, Theorem 10.3). We shall use certain quantity measuring the asymptotic bahavior of $\lambda_1(B_r)$ for certain ``balls'' $B_r$ as $r\rightarrow +\infty$, which seems to be easier to analyze. More precisely, let us first fix  a nonnegative locally Lipschitz continuous function $\rho$ on $M$, which is an exhaustion function (i.e., $B_r:=\{\rho<r\}\subset\subset{M}$ for any $r>0$), such that $|\nabla \rho|\le 1$ holds a.e.  on $M$. Note that  if $\rho$ is the distance ${\rm dist}_M(x_0,\cdot)$ from some $x_0\in M$,  then $B_r$ is precisely the geodesic ball $B(x_0,r)$. 
Define
$$
\Lambda_*:=\liminf_{r\rightarrow +\infty} \{ r^2 \lambda_1(B_r)\}.  
$$

Our main result is given as follows.

\begin{theorem}\label{th:main}
$M$ is hyperbolic if\/ $ \Lambda_* > 4t_0^2 \approx 18.624$, where $t=t_0$ is the solution to the equation
\begin{equation}\label{eq:equation_sh}
\frac{1}{4\,\mathrm{sh}^2(t/4)}+\frac{4}{\mathrm{sh}^2(t)}=1.
\end{equation}
In other words,  $\Lambda_*\leq4t_0^2$ whenever $M$ is parabolic. 
\end{theorem}

A standard example of parabolic manifolds is the plane $\mathbb R^2$,  for which $\Lambda_* =j_0^2$,  where $j_0^2\approx 5.784$ is the first zero of the Bessel function.  In view of this example and Theorem \ref{th:main},  it is of particular interest to ask the following 

\begin{problem}
What is the best lower bound for $\Lambda_*$ which implies hyperbolicity? 
\end{problem}

\begin{problem}
Does there exist a universal constant $c_0$ such that $M$ is parabolic whenever $\Lambda_*<c_0$? 
\end{problem}

We also present a simple but useful result as follow.  

\begin{proposition}\label{prop:eigenvalue}
Suppose that $\Delta \rho^2 \ge C$.
\begin{itemize}
\item[$(1)$]
If\/ $C>0$,  then
$
\Lambda_* \ge \max\left\{\frac{C}{2e},\frac{C^2}{16}\right\}.
$
\item[$(2)$]
If\/ $C>4$, then $M$ is hyperbolic.
\end{itemize}
\end{proposition}

Let us provide two applications of Proposition \ref{prop:eigenvalue} as follows. First consider a Stein manifold $M$ of complex dimension $n$, i.e., a complex manifold which admits a smooth and strictly plurisubharmonic function $\rho$.   Let $g$ be the K\"{a}hler metric given by $i\partial\bar{\partial}\rho^2$.   Since 
$
|\nabla\rho|\leq1
$
and 
$
\Delta\rho^2\geq2n,
$
it follows immediately that $M$ is hyperbolic with respect to the metric $g$ for $n\geq2$.   Analogously,  
let $M$ be a complete $n-$dimensional minimal submanifold in $\mathbb{R}^N$ and set
$
\rho(x):=\sqrt{x_1^2+\cdots+x_N^2}.
$
Then $|\nabla \rho|=1$ and $\Delta\rho^2\geq{2n}$ hold on $M$,  so that $M$ is hyperbolic for $n\geq3$.  The latter is of course known (see e.g.,  \cite{MP03} or \cite{Forstneric}),  which also indicates that the constant $4$ in Proposition \ref{prop:eigenvalue}/(2) is the best possible.  

It is also reasonable to estimate $\Lambda_*$ through volume growth conditions.  Cheng-Yau \cite{ChengYau75} showed that $\lambda_1(M)=0$ if $M$ has polynomial volume growth.   This was extended by Brooks \cite{Brooks},  who showed that 
if the volume $|M|$ of $M$ is infinite,  then
\[
\lambda_1(M) \le \frac{{\mu^*}^2}{4},\ \ \  \mu^*:=\limsup_{r\rightarrow +\infty} \frac{\log |B(x_0,r)|}{r}.
\]
The following result may be viewed as a quantitative version of the theorem of Cheng-Yau.   

\begin{theorem}\label{th:Upper}
If\/ $\nu_*:=\liminf_{r\rightarrow +\infty} \frac{\log |B_r|}{\log r}$, then
\[
{\Lambda_*} \le \inf_{0<\delta<1}  \left[\frac{\log\left((\delta^{-\nu_*}-1)^{1/2}+\delta^{-\nu_*/2}\right)}{1-\delta}\right]^2.
\]
In particular,  we have 
\begin{enumerate}
\item[$(1)$] $\Lambda_*=0$ if\/ $\nu_*=0;$
\item[$(2)$]
$
 {\Lambda_*} \le \frac{\log\left((\nu_*^{\nu_*}-1)^{1/2}+\nu_*^{\nu_*/2}\right)}{1-\nu_*}\lesssim\nu_*\log\frac{1}{\nu_*}$\/ if\/   $0<\nu_*\ll1$;
\item[$(3)$]
${\Lambda_*} \le \left[\log\left((e-1)^{1/2}+e^{1/2}\right)\right] (1+\nu_*)^2\lesssim\nu_*^2$\/ if\/ $\nu_*\gg1$.
\end{enumerate}
\end{theorem}

In \cite{Brooksfinite},  Brooks proved that if $|M|<\infty$,  then
$\lambda_1^{ess}(M) \le \frac{{\alpha^*}^2}{4}$,  where
\[
\alpha^*:=\limsup_{r\rightarrow +\infty} \frac{-\log |M\setminus B_r|}{r}.
\]
We shall show the following 

\begin{theorem}\label{th:volume_exponential}
If\/ $|M|<\infty$,  then  
\begin{equation}\label{eq:volume_exponential}
\widetilde{\Lambda}_*:=\liminf_{r\rightarrow+\infty}\frac{-\log\lambda_1(B_r)}{r}\geq \alpha_*:=\liminf_{r\rightarrow+\infty}\frac{-\log|M\setminus B_r|}{r}.
\end{equation}
\end{theorem}

Motivated by a result of Dodziuk-Pignataro-Randol-Sullivan \cite{DPRS},  we shall give examples in \S\,\ref{sec:eg} showing  that the inequalities ${\Lambda_*} \lesssim\nu_*^2$ for $\nu_*\gg1$  and $\widetilde{\Lambda}_*\geq\alpha_\ast$  are both sharp.  

\begin{problem}
Does $\widetilde{\Lambda}_*>0$ imply  $|M|<\infty$?
\end{problem}

We also provide new proofs of the theorems of Brooks mentioned above, in slightly more general forms (see \S\,\ref{sec:Brooks}).

\section{Proofs of Proposition \ref{prop:eigenvalue},  Theorem \ref{th:ess} and Corollary \ref{cor:ess}}

\begin{proof}[Proof of Proposition \ref{prop:eigenvalue}]
$(1)$ Let $\phi\in C^\infty_0(B_r)$ be fixed. It follows that
\begin{eqnarray*}
C\int_M\phi^2dV
&\leq& \int_M\phi^2\Delta\rho^2 = -\int_M\nabla\phi^2\cdot\nabla\rho^2\\
&\leq& 4\int_M\rho|\nabla\rho||\phi||\nabla\phi| \leq 4r\int_M|\phi||\nabla\phi|\\
&\leq& 4r\left(\int_M\phi^2dV\right)^{1/2}\left(\int_M|\nabla\phi|^2dV\right)^{1/2},
\end{eqnarray*}
i.e.,
\[
\frac{C^2}{16r^2}\int_M\phi^2dV\leq\int_M|\nabla\phi|^2dV.
\]
Thus $\lambda_1(B_r)\geq{C^2/16}$, which implies that
\[
\Lambda_*\geq\frac{C^2}{16}.
\]

On the other hand, let $\psi=\exp(\rho^2/2r^2)$. Clearly, $1\le \psi \le e^{1/2}$ and 
$$
\Delta \psi \ge \psi \cdot\frac{\Delta\rho^2}{2r^2} \ge \frac{C\psi}{2r^2}
$$ 
on $B_r$.  
By the following Caccioppoli-type inequality (cf. \cite{CX}, (2.4)):   
$$
\int_M \phi^2 |\nabla \psi|^2 dV +\frac1{1-\gamma} \int_M \phi^2 \psi \Delta\psi dV \le \frac1{\gamma(1-\gamma)}\int_M \psi^2 |\nabla \phi|^2 dV,\ \ \ 0<\gamma<1,
$$ 
we have 
$$
\int_M \phi^2 \psi \Delta\psi dV \le \frac1\gamma \int_M \psi^2 |\nabla \phi|^2 dV.
$$
Letting $\gamma\rightarrow 1-$, we obtain
$$
\int_M \phi^2 \psi \Delta\psi dV \le  \int_M \psi^2 |\nabla \phi|^2 dV.
$$
Thus
$$
\frac{C}{2er^2} \int_M \phi^2 dV \le \int_M |\nabla \phi|^2 dV, \ \ \ \forall\,\phi \in C^\infty_0(B_r),
$$
from which the assertion immediately follows.

$(2)$ For $\alpha>0$, we have
\begin{eqnarray*}
\Delta\rho^{-2\alpha}
&=& \alpha\left(4(\alpha+1)|\nabla\rho|^2-\Delta\rho^2\right)\rho^{-2\alpha-2}\leq\alpha\left(4(\alpha+1)-C\right)\rho^{-2\alpha-2}
\end{eqnarray*}
when $\rho\neq0$. It follows that if $0<\alpha<(C-4)/4$, then $\Delta\rho^{-2\alpha}\leq0$ for $\rho\neq0$. Let $\tau:[-\infty,0]\rightarrow[-1/2,0]$ be a smooth, convex and increasing function with $\tau\equiv-1/2$ when $-\infty\leq{t}\leq-1$ and $\tau(t)=t$ when $t\in[-1/4,0]$. Then $\tau(-\rho^{-2\alpha})$ is a non-negative subharmonic function on $M$.
\end{proof}

Recall that the capacity $\mathrm{cap}(K)$ of a compact set $K\subset M$ is given by
$$
\mathrm{cap}(K):=\inf \int_M |\nabla \psi|^2 dV,
$$
where the infimum is taken over all locally Lipschitz continuous functions $\psi$ on $M$ with a compact support such that $0\le \psi\le 1$ and $\psi|_K=1$.  The following criterion is of fundamental importance.

\begin{theorem}[cf. \cite{Grigoryan},  Theorem 5.1]\label{th:Capacity}
 $M$ is hyperbolic if and only if $\mathrm{cap}(K)>0$ for some/any compact set $K\subset M$.  
\end{theorem}

\begin{proof}[Proof of Theorem \ref{th:ess}]
Take $r_0\gg 1$ such that 
\begin{equation}\label{eq:essential2}
 \frac{\lambda_1^{ess}(M)}2 \int_M \phi^2 dV \le \int_M |\nabla \phi|^2 dV
\end{equation}
holds for any 
  locally Lipschitz,  compactly supported function  $\phi$ on $M\setminus B_{r_0}$. Let $\psi $ be a locally Lipschitz,  compactly supported function on $M$.  Choose a cut-off function $\eta:M\rightarrow [0,1]$ such that $\eta=1$ for $\rho\ge r_0+1$,  $\eta=0$ for $\rho\le r_0$ and $|\nabla \eta| \le 1$.  Apply \eqref{eq:essential2} with $\phi=\eta\psi$,  we have 
\begin{eqnarray}\label{eq:essential3}
\frac{\lambda_1^{ess}(M)}2 \int_{ \rho\ge r_0+1 } \psi^2  dV & \le &  \int_M |\nabla (\eta \psi)|^2  dV\nonumber\\
& \le & 2 \int_M |\nabla  \psi|^2  dV + 2 \int_{B_{r_0+1}} \psi^2  dV.
\end{eqnarray}
Since $M$ has infinite volume,  
 we may take $r_1 > r_0+1$ such that 
 $$
 \frac{\lambda_1^{ess}(M)}2 |B_{r_1}\setminus B_{r_0+1}| > 2 |B_{r_0+1} | + 2.
 $$
Thus if $\psi=1$ on $\overline{B}_{r_1}$,   then it follows from \eqref{eq:essential3} that
$$
\int_M |\nabla \psi|^2 dV >  1,
$$
so that 
$$
\mathrm{cap}(\overline{B}_{r_1}) \ge 1
$$
and $M$ is hyperbolic in view of Theorem \ref{th:Capacity}.
\end{proof}

\begin{proof}[Proof of Corollary \ref{cor:ess}]
The {\it only if}\/ part is trivial,  since near punctures,  $g_{\rm hyp}$ is equivalent to the hyperbolic metric of the punctured disc,  which is  $d-$bounded near the puncture.   For the {\it if}\/ part,  first observe that  $\lambda_1^{ess}(\widetilde{M})>0$,  in view of the proof of Theorem 1.4.A in Gromov \cite{Gromov91}.  Since quasi-isometry preserves the type (see \cite{Grigoryan},  Corollary 5.3),  so $\widetilde{M}$ is also parabolic.   By Theorem \ref{th:ess}, we conclude that $(\widetilde{M},g)$ has finite  volume,  so does $(M,g_{\rm hyp})$,  since quasi-isometry also preserves volume growth,  which in turn implies the conformal finiteness of $M$.  
\end{proof}

\section{Proof of Theorem \ref{th:main}}\label{sec:main}
We start with a technical lemma as follows. Given $A>0$, define
\begin{equation}\label{eq:functional_J}
J_\chi(t):=\chi'(t)^2-A^2\chi(t)^2.
\end{equation}

\begin{lemma}\label{lm:J}
Among all $C^1$ functions $\chi:[a,b]\rightarrow[0,+\infty)$ with $\chi(a)=0$ and $\chi(b)=1$, the functional
\[
\chi\mapsto\sup_{t\in[a,b]}J_\chi(t)
\]
acheives its minimum at
\[
\chi_0(t)=\frac{e^{A(t-a)}-e^{-A(t-a)}}{e^{A(b-a)}-e^{-A(b-a)}}=\frac{\mathrm{sh}(A(t-a))}{\mathrm{sh}(A(b-a))},
\]
with
\begin{equation}\label{eq:minimum_J}
J_{\chi_0}(t)\equiv\frac{4A^2}{\big(e^{A(b-a)}-e^{-A(b-a)}\big)^2}=\frac{A^2}{\mathrm{sh}^2(A(b-a))}.
\end{equation}
\end{lemma}

\begin{proof}
A straightforward calculation immediately yields \eqref{eq:minimum_J}. Now suppose on the contrary that
\[
\sup_{t\in[a,b]}J_\chi(t)<\sup_{t\in[a,b]}J_{\chi_0}(t)
\]
for some $C^1$ function $\chi$ on $[a,b]$ with $\chi\geq0$, $\chi(a)=0$ and $\chi(b)=1$. First note that there exists some $\delta>0$ with
\[
\chi(t)<\chi_0(t),\ \ \ \forall\,a<{t}\leq{a+\delta},
\]
for otherwise $\chi'(a)\geq\chi_0'(a)>0$, so that
\[
\sup_{t\in[a,b]}J_\chi(t)\geq{J_\chi(a)}\geq\chi'(a)^2\geq\chi_0'(a)^2=J_{\chi_0}(a)=\sup_{t\in[a,b]}J_{\chi_0}(t),
\]
which is absurd. Set
\[
c:=\sup\{t\in[a,b]:\chi(s)<\chi_0(s),\ \forall\,s\in(a,t]\}.
\]
It follows that $c>a$, $\chi(c)=\chi_0(c)$ and $\chi(t)<\chi_0(t)$ for all $a<t<c$. Thus there exists some $t_1\in(a,c)$, according to Cauchy's intermediate value theorem, such that
\[
\frac{\chi'(t_1)}{\chi_0'(t_1)}=\frac{\chi(c)-\chi(a)}{\chi_0(c)-\chi_0(a)}=1.
\]
However,
\[
\chi'(t_1)^2-A^2\chi(t_1)^2\leq\sup_{t\in[a,b]}J_\chi(t)<\sup_{t\in[a,b]}J_{\chi_0}(t)=\chi_0'(t_1)^2-A^2\chi_0(t_1)^2,
\]
so that $\chi(t_1)>\chi_0(t_1)$, which is impossible.
\end{proof}

We shall prove  a slightly more general result as follows.  

 \begin{theorem}\label{th:main_2}
Let $t_0$ be the solution to \eqref{eq:equation_sh}. Suppose the following conditions hold:
\begin{enumerate}
\item[$(1)$] there exists a numerical constant $C_0>4t_0^2\approx18.624$ such that 
\begin{equation}\label{eq:growth2}
\lambda_1(B_{r}\setminus \overline{B}_{r/8}) \ge \frac{C_0}{r^2},\ \ \ \forall\,r\gg 1;
\end{equation}
\item[$(2)$] $\int_M\frac{dV}{1+\rho^2}=+\infty$.
\end{enumerate}
Then $M$ is hyperbolic.  
\end{theorem}

\begin{proof}
 Let $\psi$ be any fixed locally Lipschitz,  compactly supported function on $M$.  Take a Lipschitz function $\chi:\mathbb{R}\rightarrow [0,1]$ such that $\chi(t)=1$ for $1/2\le t \le 1$ and $\phi=0$ for $t\ge 2$ or $t\le 1/4$. For $\phi:=\chi(\rho/r)$, we have
\begin{eqnarray}\label{eq:annuli_lower}
\int_M |\nabla(\phi\psi)|^2 dV
& \ge & \lambda_1(B_{2r}\setminus \overline{B}_{r/4}) \int_M \phi^2\psi^2 dV\nonumber\\
& \ge & \frac{C_0}{4r^2}\int_{r/4\leq\rho\leq{r/2}}\chi\left(\rho/r\right)^2\psi^2dV_g\nonumber\\
& & +\frac{C_0}{4r^2}\int_{r\leq\rho\leq{2r}}\chi\left(\rho/r\right)^2\psi^2dV_g\nonumber\\
& & +\frac{C_0}{4r^2} \int_{r/2\le \rho\le r} \psi^2 dV
\end{eqnarray}
for all $r\gg 1$.  On the other hand, for any $\gamma>0$, we have
\begin{eqnarray*}
\int_M |\nabla(\phi\psi)|^2 dV
& \le & (1+\gamma)\int_M \psi^2 |\nabla \phi|^2 dV + (1+1/\gamma) \int_M \phi^2 |\nabla \psi|^2 dV\\
& \le & \frac{1+\gamma}{r^2} \int_{r/4\le \rho\le r/2} \chi'(\rho/r)^2\psi^2 dV\nonumber\\
& & + \frac{1+\gamma}{r^2}\int_{r\le \rho\le 2r} \chi'(\rho/r)^2\psi^2 dV \\
& & + (1+1/\gamma) \int_{r/4\le \rho\le 2r} |\nabla \psi|^2 dV.
\end{eqnarray*}
This together with \eqref{eq:annuli_lower} yield
\begin{eqnarray}\label{eq:annuli_0}
\frac{C_0}{4r^2}\int_{r/2\le \rho\le r} \psi^2 dV
&\le & \frac{1+\gamma}{r^2} \int_{r/4\le \rho\le r/2} J_\chi(\rho/r)\psi^2 dV \nonumber\\
& & + \frac{1+\gamma}{r^2}\int_{r\le \rho\le 2r} J_\chi(\rho/r)\psi^2 dV \nonumber\\
& & + (1+1/\gamma) \int_{r/4\le \rho\le 2r} |\nabla \psi|^2 dV, 
\end{eqnarray}
where $J_\chi$ is the function defined in \eqref{eq:functional_J} with
\[
A:=\frac{1}{2}\left(\frac{C_0}{1+\gamma}\right)^{1/2}.
\]
Motivated by Lemma \ref{lm:J}, we set
\[
\chi(t)=
\begin{cases}
0,\ \ \ &t\leq1/4,\\
\chi_1(t),\ \ \ &1/4\leq{t}\leq1/2,\\
1,\ \ \ &1/2\leq{t}\leq1,\\
\chi_2(t),\ \ \ &1\leq{t}\leq2,\\
0,\ \ \ &t\geq2,
\end{cases}
\]
where
\[
\chi_1(t):=\frac{e^{A(t-1/4)}-e^{-A(t-1/4)}}{e^{A/4}-e^{-A/4}}\ \ \ \text{and}\ \ \ \chi_2(t):=\frac{e^{A(2-t)}-e^{-A(2-t)}}{e^{A}-e^{-A}}.
\]
It follows from \eqref{eq:minimum_J} that
\begin{equation}\label{eq:upper_J}
J_\chi(t)\leq
\begin{cases}
\frac{A^2}{\mathrm{sh}^2(A/4)},\ \ \ &1/4\leq{t}\leq1/2,\\
\frac{A^2}{\mathrm{sh}^2(A)},\ \ \ &1\leq{t}\leq2.
\end{cases}
\end{equation}
By \eqref{eq:annuli_0} and \eqref{eq:upper_J}, we obtain
\begin{eqnarray*}
\frac{1}{r^2}\int_{r/2\le \rho\le r} \psi^2 dV
&\le & \frac{1}{\mathrm{sh}^2(A/4)r^2} \int_{r/4\le \rho\le r/2} \psi^2 dV\\
& & + \frac{1}{\mathrm{sh}^2(A)r^2}\int_{r\le \rho\le 2r} \psi^2 dV \\
& & + \frac{4(1+1/\gamma)}{C_0} \int_{r/4\le \rho\le 2r} |\nabla \psi|^2 dV, 
\end{eqnarray*}
In particular, if we take $r=2^k$, then
\begin{eqnarray}\label{eq:annuli_1}
\frac{1}{2^{2k}}\int_{2^{k-1}\le \rho\le 2^k} \psi^2 dV
&\le & \frac{1}{4\,\mathrm{sh}^2(A/4)}\cdot\frac{1}{2^{2k-2}}\int_{2^{k-2}\le \rho\le 2^{k-1}} \psi^2 dV\nonumber\\
& & + \frac{4}{\mathrm{sh}^2(A)}\cdot\frac{1}{2^{2k+2}}\int_{2^k\le \rho\le 2^{k+1}} \psi^2 dV \nonumber\\
& & + \frac{4(1+1/\gamma)}{C_0} \int_{2^{k-2}\le \rho\le 2^{k+1}} |\nabla \psi|^2 dV.  
\end{eqnarray}
for all  integers $k\ge k_0\gg 1$. By setting
\[
A_k:=\frac{1}{2^{2k}}\int_{2^{k-1}\le \rho\le 2^k} \psi^2 dV,
\]
we may rewrite \eqref{eq:annuli_1} as
\[
A_k \leq \frac{A_{k-1}}{4\,\mathrm{sh}^2(A/4)}+\frac{4A_{k+1}}{\mathrm{sh}^2(A)}+\frac{4(1+1/\gamma)}{C_0}\int_{2^{k-2}\le \rho\le 2^{k+1}} |\nabla \psi|^2 dV.
\]
Take sum $\sum_{k=k_0}^\infty$,  we get
\begin{eqnarray*}
\sum^\infty_{k=k_0}A_k
&\leq & \frac{1}{4\,\mathrm{sh}^2(A/4)}\sum^\infty_{k=k_0}A_{k-1}+\frac{4}{\mathrm{sh}^2(A)}\sum^\infty_{k=k_0}A_{k+1}\\
& & +12(1+1/\gamma)\int_M |\nabla \psi|^2 dV\\
&\leq & \left(\frac{1}{4\,\mathrm{sh}^2(A/4)}+\frac{4}{\mathrm{sh}^2(A)}\right)\sum^\infty_{k=k_0}A_k+\frac{A_{k_0-1}}{4\,\mathrm{sh}^2(A/4)}\\
& & +\frac{12(1+1/\gamma)}{C_0}\int_M |\nabla \psi|^2 dV,
\end{eqnarray*}
i.e.,
\begin{equation}\label{eq:annuli_sum}
g(A)\sum^\infty_{k=k_0}A_k\le  \frac{A_{k_0-1}}{4\,\mathrm{sh}^2(A/4)}+\frac{12(1+1/\gamma)}{C_0}\int_M |\nabla \psi|^2 dV,
\end{equation}
where
\[
g(A):=1-\frac{1}{4\,\mathrm{sh}^2(A/4)}-\frac{4}{\mathrm{sh}^2(A)}.
\]
Note that $g(t)$ is strictly increasing when $t>0$ and $t=t_0$ is the unique zero of $g$. Moreover, if $C_0>4t_0^2$, then we may choose $0<\gamma\ll1$ so that
\[
A=\frac{1}{2}\left(\frac{C_0}{1+\gamma}\right)^{1/2}>t_0.
\]
Thus
\[
g(A)>g(t_0)=0.
\]

Finally, we assume that $\psi=1$ when $\rho\leq2^l$, where $l\gg{k_0}$. It follows that
\[
\sum^{+\infty}_{k=k_0}A_k\geq\sum^l_{k=k_0}\frac{|B_{2^k}\setminus B_{2^{k-1}}|}{2^{2k}}.
\]
Clearly, the second condition in the theorem is equivalent to
\[
\sum^{+\infty}_{k=0}\frac{|B_{2^k}\setminus B_{2^{k-1}}|}{2^{2k}}=+\infty.
\]
It follows that if $l\gg{k_0}$, then 
\[
g(A)\sum^l_{k=k_0}\frac{|B_{2^k}\setminus B_{2^{k-1}}|}{2^{2k}} - \frac{A_{k_0-1}}{4\,\mathrm{sh}^2(A/4)}  > 1.
\]
These together with \eqref{eq:annuli_sum} give
$$
\int_M |\nabla \psi|^2 dV >  \frac{C_0\gamma}{12(1+\gamma)}
$$
for all locally Lipschitz, compactly supported function $\psi$ on $M$ with $\psi=1$ on $B_{2^l}$, which implies
$$
\mathrm{cap}(\overline{B}_{2^l}) \ge \frac{C_0\gamma}{12(1+\gamma)}.
$$
Thus $M$ is hyperbolic in view of Theorem \ref{th:Capacity}.
\end{proof}

\begin{corollary}\label{cor:main_1}
Let $t_0$ be the solution to \eqref{eq:equation_sh}. Suppose the following conditions hold:
\begin{enumerate}
\item[$(1)$] there exists a numerical constant $C_0>4t_0^2\approx18.623$ such that \eqref{eq:growth2} hold.

\item[$(2)$] $\int^{+\infty}_1\frac{v(r)}{r^3}dr=+\infty$, where $v(r):=|B_r|=|\{\rho<r\}|$.
\end{enumerate}
Then $M$ is hyperbolic.  
\end{corollary}

\begin{proof}
By the coarea formula, we have
\[
v(r)=\int^r_0\left(\int_{\{\rho=t\}}\frac{1}{|\nabla\rho|}\right)dt,\ \ \ v'(r)=\int_{\{\rho=r\}}\frac{1}{|\nabla\rho|},
\]
and
\[
\int_M\frac{dV}{1+\rho^2}=\int^{+\infty}_0\frac{v'(r)}{1+r^2}dr=\left.\frac{v(r)}{1+r^2}\right|^{+\infty}_0+\int^{+\infty}_0\frac{2rv'(r)}{(1+r^2)^2}dr.
\]
Thus Theorem \ref{th:main_2} applies.
\end{proof}

\begin{proof}[Proof of Theorem \ref{th:main}]
In view of Theorem \ref{th:main_2},  it suffices to verify the following lemma.
\end{proof}

\begin{lemma}\label{lm:infinity}
Suppose there exists a numerical constant $C_1> 4(\log(2+\sqrt{3}))^2\approx6.938$ such that 
$$
\lambda_1(B_r) \ge C_1/r^2,\ \ \ \forall\,r\gg 1.
$$
Then 
\[
\int_M\frac{dV}{1+\rho^2}=+\infty.
\]
\end{lemma}

\begin{proof}
It suffices to verify
$$
\sum^{+\infty}_{k=1}\frac{|B_{2^k}\setminus B_{2^{k-1}}|}{2^{2k}}=+\infty.
$$
Let $\chi:\mathbb R\rightarrow [0,1]$ be a cut-off function such that $\chi|_{(-\infty,1]}=1$,  $\chi|_{[2,+\infty)} = 0$ and
\[
\chi(t)=\frac{e^{\sqrt{C_1}(2-t)/2}-e^{-\sqrt{C_1}(2-t)/2}}{e^{\sqrt{C_1}/2}-e^{-\sqrt{C_1}/2}},\ \ \ t\in[1,2].
\]
Set $\phi=\chi(\rho/r)$. Then we have
\begin{eqnarray}
\int_M |\nabla \phi|^2 dV  & \ge & \lambda_1(B_{2r}) \int_M \phi^2 dV\nonumber\\
& \ge & \frac{C_1}{4 r^2} \cdot |B_r|+\frac{C_1}{4r^2}\int_{r\leq\rho\leq2r}\chi(\rho/r)^2dV\label{eq:inf_1}
\end{eqnarray}
for all $r\gg 1$.  On the other hand,  since $|\nabla \rho|\le 1$,  we have
\begin{equation}\label{eq:inf_2}
\int_M |\nabla \phi|^2 dV \le \frac{1}{r^2}\int_{r\leq\rho\leq2r}\chi'(\rho/r)^2dV.
\end{equation}
Thus
\[
\frac{C_1}{4}|B_r|\leq\int_{r\leq\rho\leq2r}J_\chi(\rho/r)dV_g,
\]
where $J_\chi$ is the function given by \eqref{eq:functional_J} with $A=\sqrt{C_1}/2$. By \eqref{eq:minimum_J}, we have
\[
J_\chi(t)\equiv\frac{C_1}{(e^{\sqrt{C_1}/2}-e^{-\sqrt{C_1}/2})^2}=\frac{C_1}{4\,\mathrm{sh}^2(\sqrt{C_1}/2)},
\]
so that
\[
|B_r|\leq\frac{|B_{2r}|-|B_r|}{\mathrm{sh}^2(\sqrt{C_1}/2)},
\]
i.e.,
\[
|B_{2r}|\geq\left(1+\frac{1}{\mathrm{sh}^2(\sqrt{C_1}/2)}\right)|B_r|=:C_2|B_r|.
\]
In particular, we have
\[
|B_{2^k}|\geq{C_2^{k-k_0}}|B_{2^k_0}|,
\]
for all $k\geq{k_0}\gg1$, so that
\[
|B_{2^k}\setminus B_{2^{k-1}}|\geq\left(1-\frac{1}{C_2}\right)|B_{2^k}|\geq\left(1-\frac{1}{C_2}\right){C_2^{k-k_0}}|B_{2^k_0}|.
\]
Thus
\[
\sum^{+\infty}_{k=1}\frac{|B_{2^k}\setminus B_{2^{k-1}}|}{2^{2k}}=+\infty
\]
provided $C_2>4$, i.e., $C_1>4(\log(2+\sqrt{3}))^2$.
\end{proof}

\section{Proofs of Theorem \ref{th:Upper} and Theorem \ref{th:volume_exponential}}

\begin{proof}[Proof of Theorem \ref{th:Upper}]

For $0<\varepsilon\ll 1$, we take $r_\varepsilon\gg 1$ such that 
$$
\lambda_1(B_r)\ge \frac{\Lambda_*-\varepsilon}{r^2}, \ \ \ r\ge r_\varepsilon. 
$$
Let $r\ge r_\varepsilon$ and $0<\delta<1$. Take a cut-off function $\chi:\mathbb{R}\rightarrow[0,1]$ such that $\chi|_{(-\infty,\delta]}=1$,  $\chi|_{[1,+\infty)}=0$ and
\[
\chi(t):=\frac{e^{\sqrt{\Lambda_*-\varepsilon}(1-t)}-e^{-\sqrt{\Lambda_*-\varepsilon}(1-t)}}{e^{\sqrt{\Lambda_*-\varepsilon}(1-\delta)}-e^{-\sqrt{\Lambda_*-\varepsilon}(1-\delta)}},\ \ \ t\in[\delta,1].
\]
Set $\phi=\chi(\rho/r)$. Then we have
\begin{eqnarray*}
\frac{\Lambda_*-\varepsilon}{r^2}\,|B_{\delta{r}}|
&=& \frac{\Lambda_*-\varepsilon}{r^2} \int_M \phi^2 dV - \frac{\Lambda_*-\varepsilon}{r^2} \int_{\rho\geq\delta{r}} \phi^2 dV\\
&\leq& \int_M |\nabla\phi|^2 dV - \frac{\Lambda_*-\varepsilon}{r^2} \int_{\rho\geq\delta{r}} \phi^2 dV\\
&\leq& \frac{1}{r^2}\int_{\delta{r}\leq\rho\leq{r}} \left(\chi'(\rho/r)^2 - (\Lambda_*-\varepsilon)\chi(\rho/r)^2\right) dV\\
&\leq& \frac{\Lambda_*-\varepsilon}{r^2\mathrm{sh}^2\left(\sqrt{\Lambda_*-\varepsilon}(1-\delta)\right)}\big(|B_r| - |B_{\delta{r}}|\big),
\end{eqnarray*}
in view of Lemma \ref{lm:J}. Namely,
$$
|B_r| \ge \left( 1+ \mathrm{sh}^2\big(\sqrt{\Lambda_*-\varepsilon}(1-\delta)\big)\right) |B_{\delta r}|.  
$$
In particular, if $k\ge k_{\varepsilon,\delta}\gg 1$, then
$$
|B_{\delta^{-k}}| \ge \left( 1+ \mathrm{sh}^2\big(\sqrt{\Lambda_*-\varepsilon}(1-\delta)\big)\right) ^{k-k_{\varepsilon,\delta}} \left|B_{\delta^{-k_{\varepsilon,\delta}}}\right|.
$$
Since $|B_r|\geq|B_{\delta^{-k}}|$ and $\log{r}\leq-(k+1)\log\delta$ whenever $\delta^{-k}\leq{r}\leq\delta^{-k-1}$, we have
$$
\nu_* \geq \liminf_{k\rightarrow +\infty} \frac{\log |B_{\delta^{-k}}|}{-(k+1)\log \delta}\ge  \frac{ \log \left( 1+ \mathrm{sh}^2\big(\sqrt{\Lambda_*-\varepsilon}(1-\delta)\big)\right)}{-\log \delta}.
$$
Thus
$$
\Lambda_*-\varepsilon \le \left(\frac{\log\left((\delta^{-\nu_*}-1)^{1/2}+\delta^{-\nu_*/2}\right)}{1-\delta}\right)^2,\ \ \ \forall\,\delta\in(0,1).
$$
Since $\varepsilon$ can be arbitrarily small, the first assertion immediately follows, which in turn immediately implies that $\Lambda_*=0$ if $\nu_*=0$. To verify (2) and (3), it suffices to take $\delta=\nu_*/(1+\nu_*)$ and $\delta=\nu_*$, respectively.
\end{proof}

\begin{proof}[Proof of Theorem \ref{th:volume_exponential}]
By definition, there exists a sequence $\{r_k\}$ with $\lim_{k\rightarrow+\infty}r_k=+\infty$, such that $\lambda_1(B_{r_k})>e^{-(\beta+\varepsilon)r_k}$ for some $0<\varepsilon\ll1$. Again, for $k\geq1$ and $0<\delta<1$, we take a cut-off function $\chi_k:\mathbb{R}\rightarrow[0,1]$ such that $\chi_k|_{(-\infty,\delta]}=1$,  $\chi_k|_{[1,+\infty)}=0$ and
\[
\chi_k(t):=\frac{e^{A_k(1-t)}-e^{-A_k(1-t)}}{e^{A_k(1-\delta)}-e^{-A_k(1-\delta)}},\ \ \ t\in[\delta,1],
\]
where
\[
A_k=\frac{r_k}{e^{(\beta+\varepsilon)r_k/2}}.
\]
Set $\phi_k=\chi_k(\rho/r_k)$. Then
\begin{eqnarray*}
e^{-(\beta+\varepsilon)r_k}\,|B_{\delta{r_k}}|
&=& e^{-(\beta+\varepsilon)r_k} \int_M \phi_k^2 dV - e^{-(\beta+\varepsilon)r_k}\int_{\rho\geq\delta{r_k}} \phi_k^2 dV\\
&\leq& \int_M |\nabla\phi_k|^2 dV - e^{-(\beta+\varepsilon)r_k}\int_{\rho\geq\delta{r_k}} \phi_k^2 dV\\
&\leq& \frac{1}{r_k^2}\int_{\delta{r_k}\leq\rho\leq{r_k}} \left(\chi_k'(\rho/r_k)^2 - A_k^2\chi_k(\rho/r_k)^2\right) dV\\
&\leq& \frac{A_k^2}{r_k^2\,\mathrm{sh}^2\left(A_k(1-\delta)\right)}|B_{r_k}\setminus B_{\delta{r_k}}|\\
&\leq& \frac{A_k^2}{r_k^2\,\mathrm{sh}^2\left(A_k(1-\delta)\right)}|M\setminus B_{\delta{r_k}}|.
\end{eqnarray*}
in view of Lemma \ref{lm:J}. That is,
\[
|M|\geq\left(1+\mathrm{sh}^2\left(A_k(1-\delta)\right)\right)|B_{\delta{r_k}}|,
\]
which is equivalent to
\[
|M\setminus{B_{\delta{r_k}}}|\geq\frac{\mathrm{sh}^2\left(A_k(1-\delta)\right)}{1+\mathrm{sh}^2\left(A_k(1-\delta)\right)}|M|.
\]
Since $\mathrm{sh}^2\left(A_k(1-\delta)\right)\sim{A_k^2}(1-\delta)^2$ as $k\rightarrow+\infty$, we have
\[
\alpha\leq\lim_{k\rightarrow\infty}\frac{-\log|M\setminus B_{\delta{r_k}}|}{\delta{r_k}}=\frac{\beta+\varepsilon}{\delta}.
\]
Letting $\delta\rightarrow1-$ and $\varepsilon\rightarrow0+$, we conclude that $\beta\geq\alpha$.
\end{proof}

\section{New proofs of Brooks' theorems}\label{sec:Brooks}

In this section,  we provide alternative proofs for Brooks' theorems, in slightly more general settings.   

\begin{theorem}\label{th:BrooksNew}
$$
\lambda_1(M) \le \frac{\mu_*^2}{4},\ \ \ \mu_*:=\liminf_{r\rightarrow +\infty} \frac{\log |B_r|}r.
$$
\end{theorem} 

\begin{proof}
Let $\phi$ be a
  locally Lipschitz,  compactly supported function   on $M$.   For any  $0<\lambda< \sqrt{\lambda_1(M)} $,  we have 
$$
\sqrt{\lambda_1(M)} \| e^{-\lambda\rho} \phi\|    \le   \| \nabla(e^{-\lambda \rho} \phi) \| 
 \le  \lambda \| e^{-\lambda \rho}\phi\| + \| e^{-\lambda\rho} \nabla \phi\|,
$$
i.e.,  
\begin{equation}\label{eq:B1}
\beta  \| e^{-\lambda\rho} \phi\| \le \| e^{-\lambda\rho} \nabla \phi\|,  \ \ \ \beta:=\sqrt{\lambda_1(M)}-\lambda. 
\end{equation}
Given $r> 1$,  choose a cut-off function $\eta_r:M\rightarrow [0,1]$ such that $\eta_r=1$ for $\rho\le r-1$,  $\eta_r=0$ for $\rho\ge r$ and $|\nabla \eta_r| \le 1$.  Consider the test function $\phi=e^{\lambda r} \eta_r$.  We have 
$$
\| e^{-\lambda\rho} \nabla \phi\|^2\le e^2  |B_r\setminus B_{r-1}|,
$$ 
while for any $0<\varepsilon<1$ and $r\ge \frac1{1-\varepsilon}$,  
$$
\| e^{-\lambda\rho} \phi\|^2 \ge \int_{\rho\le \varepsilon r} e^{2\lambda r-2 \lambda \rho} dV\ge e^{2(1-\varepsilon)\lambda r} |B_{\varepsilon r}|.
$$
These together with \eqref{eq:B1} yield
\begin{equation}\label{eq:B_2}
|B_{r_k}| \ge (\beta/e)^2 e^{2(1-\varepsilon)\lambda r} |B_{\varepsilon r}|.
\end{equation}
Suppose on the contrary that $\lambda_1(M)>\mu_*^2/4$.  Then there exist $0<\alpha<1$ and a sequence $r_k\rightarrow +\infty$ such that
$$
|B_{r_k}| \le e^{2\alpha \sqrt{\lambda_1(M)}r_k}.
$$
But this contradicts  \eqref{eq:B_2} provided $(1-\varepsilon)\lambda>\alpha \sqrt{\lambda_1(M)}$.  
\end{proof}

\begin{theorem}\label{th:BrooksNew_finite}
If $|M|<\infty$, then
\[
\lambda_1^{ess}(M) \le \frac{\alpha_*^2}{4},\ \ \ \alpha_*:=\liminf_{r\rightarrow +\infty} \frac{-\log |M\setminus B_r|}{r}.
\]
\end{theorem} 

\begin{proof}
For any $\varepsilon>0$, we have $|M\setminus{B_r}|\geq{e^{-(\alpha_*-\varepsilon)r}}$ when $r\gg1$. Let $R\gg{r}$. Choose a cut-off function $\eta_{r,R}:M\rightarrow[0,1]$ such that $\eta_{r,R}=0$ for $\rho\leq{r}$ and $\rho\geq{R+1}$, $\eta_{r,R}=1$ for $r+1\leq\rho\leq{R}$ and $|\nabla\eta_{r,R}|\leq1$. Set $\phi:=e^{(\alpha_*+\varepsilon)\rho/2}\eta_{r,R}$. It follows that
\begin{equation}\label{eq:Brooks_finite_1}
\int_{M\setminus{B_r}}\phi^2dV\geq\int_{r+1\leq\rho\leq{R}}e^{(\alpha_*+\varepsilon)\rho}dV
\end{equation}
and
\begin{eqnarray}\label{eq:Brooks_finite_2}
\int_{M\setminus{B_r}}|\nabla\phi|^2dV
&=& \int_{M\setminus{B_r}}\left|\frac{\alpha_*+\varepsilon}{2}e^{(\alpha_*+\varepsilon)\rho/2}\eta_{r,R}\nabla\rho+e^{(\alpha_*+\varepsilon)\rho/2}\nabla\eta_{r,R}\right|^2dV\nonumber\\
&\leq& \frac{(1+\delta)(\alpha_*+\varepsilon)^2}{4}\int_{M\setminus{B_r}}\phi^2dV\nonumber\\
& & +\left(1+\frac{1}{\delta}\right)\int_{M\setminus{B_r}}e^{(\alpha_*+\varepsilon)\rho}|\nabla\eta_{r,R}|^2dV,
\end{eqnarray}
where $\delta>0$ and
\begin{equation}\label{eq:Brooks_finite_3}
\int_{M\setminus{B_r}}e^{(\alpha_*+\varepsilon)\rho}|\nabla\eta_{r,R}|^2dV\leq\int_{r\leq\rho\leq{r+1}}e^{(\alpha_*+\varepsilon)\rho}dV+\int_{R\leq\rho\leq{R+1}}e^{(\alpha_*+\varepsilon)\rho}dV.
\end{equation}
For simplicity, we define
\[
F(t):=\int_{r\leq\rho\leq{t}}e^{(\alpha_*+\varepsilon)\rho}dV.
\]
It follows from \eqref{eq:Brooks_finite_1}-\eqref{eq:Brooks_finite_3} that
\begin{equation}\label{eq:Brooks_finite_lambda}
\lambda_1(M\setminus{B_r}) \leq \frac{(1+\delta)(\alpha_*+\varepsilon)^2}{4}+\left(1+\frac{1}{\delta}\right)\frac{F(r+1)+F(R+1)-F(R)}{F(R)-F(r+1)}.
\end{equation}
Take a sequence $\{r_k\}$ which increases to $+\infty$ such that $|M\setminus{B_{r_k}}|\geq{e^{-(\alpha_*+\varepsilon)r_k}}$ when $k\gg1$. Thus
\[
\int_{\rho\geq{r_k}}e^{(\alpha_*+\varepsilon)\rho}dV\geq{e^{(\alpha_*+\varepsilon)r_k}}|M\setminus{B_{r_k}}|\geq1,
\]
so that $\int_Me^{(\alpha_*+\varepsilon)\rho}dV=+\infty$, i.e., $\lim_{R\rightarrow+\infty}F(R)=+\infty$.

We claim that there exists a sequence $\{m_k\}$ of positive integers which increases to $+\infty$, such that
\begin{equation}\label{eq:Brooks_finite_m_k}
F(m_k+1)\leq{e^{2\varepsilon}}F(m_k).
\end{equation}
Otherwise $F(m+1)>e^{c}F(m)$ when $m\gg1$ for some $c>2\varepsilon$, so that $F(m)\gtrsim{e^{cm}}$. Thus
\begin{equation}\label{eq:Brooks_finite_F}
F(m+1)-F(m)>(e^c-1)F(m)\gtrsim{e^{cm}}.
\end{equation}
Here and in what follows in this section, the implicit constants are independent of $m$. On the other hand, we have
\begin{eqnarray*}
F(m+1)-F(m)
&=& \int_{m\leq\rho\leq{m+1}}e^{(\alpha_*+\varepsilon)\rho}dV\\
&\leq& e^{(\alpha_*+\varepsilon)(m+1)}|M\setminus{B_m}|\\
&\leq& e^{(\alpha_*+\varepsilon)(m+1)-(\alpha_*-\varepsilon)m}\\
&\lesssim& e^{2\varepsilon{m}},
\end{eqnarray*}
which is impossible, for $2\varepsilon<c$. Thus \eqref{eq:Brooks_finite_m_k} holds for some sequence $\{m_k\}$, so that
\[
\limsup_{k\rightarrow+\infty}\frac{F(r+1)+F(m_k+1)-F(m_k)}{F(m_k)-F(r+1)}\leq{e^{2\varepsilon}-1}.
\]
This together with \eqref{eq:Brooks_finite_lambda} give
\[
\lambda_1(M\setminus{B_r}) \leq \frac{(1+\delta)(\alpha_*+\varepsilon)^2}{4}+\left(1+\frac{1}{\delta}\right)(e^{2\varepsilon}-1).
\]
Letting first $\varepsilon\rightarrow0+$ and then $\delta\rightarrow0+$, we conclude that $\lambda_1(M\setminus{B_r})\leq\alpha_*^2/4$, from which the assertion immediately follows.
\end{proof}

\section{Examples}\label{sec:eg}

Let $M=\mathbb{R}\times{S^1}$ be equipped with the followsing Riemannian metric
\[
g=dt^2+\eta'(t)^2d\theta^2,\ \ \ t\in\mathbb{R},\ e^{i\theta}\in{S^1},
\]
where $\eta:\mathbb{R}\rightarrow\mathbb{R}$ is a smooth function such that $\eta'(t)>0$ and $\lim_{t\rightarrow-\infty}\eta(t)=0$. Dodziuk-Pigmataro-Randol-Sullivan \cite[Proposition 3.1]{DPRS} showed that if $\eta(t)=e^t$, then $\lambda_1(M)\geq1/4$.

Let $\rho(t,\theta)=|t|$. Clearly, $\rho$ is an exhaustion function which satisfies $|\nabla\rho|_g\leq1$. The goal of this section is to investigate the asymptotic behavior of $\lambda_1(B_r)$ as $r\rightarrow+\infty$ for different choices of $\eta$. We start with the following elementary lower estimate, .

\begin{proposition}\label{prop:example}
\[
\lambda_1(B_r)\geq\frac{1}{4}\inf_{|t|\leq{r}}\frac{\eta'(t)^2}{\eta(t)^2}.
\]
\end{proposition}

\begin{proof}
The idea is essentially implicit in \cite{DPRS}. Since $dV=\eta'(t)dtd\theta$, we have
\[
\int^r_{-r}\phi^2\eta'(t)dt=2\int^r_{-r}\phi\frac{\partial\phi}{\partial{t}}\eta(t)dt,\ \ \ \forall\,\phi\in{C^\infty_0(B_r)},
\]
so that
\begin{eqnarray}\label{eq:Cauchy_Schwarz_example}
\int^r_{-r}\phi^2\eta'(t)dt
&\leq& 4\int^r_{-r}\left|\frac{\partial\phi}{\partial{t}}\right|^2\frac{\eta(t)^2}{\eta'(t)}dt\leq4\int^r_{-r}\left|\nabla\phi\right|^2\frac{\eta(t)^2}{\eta'(t)}dt\nonumber\\
&\leq& 4\sup_{|t|\leq{r}}\frac{\eta(t)^2}{\eta'(t)^2}\int^r_{-r}|\nabla\phi|^2\eta'(t)dt
\end{eqnarray}
in view of the Cauchy-Schwarz inequality. Thus
\[
\int_{B_r}\phi^2dV
= \int^{2\pi}_0\int^r_{-r}\phi^2\eta'(t)dtd\theta \leq 4\sup_{|t|\leq{r}}\frac{\eta(t)^2}{\eta'(t)^2}\int_{B_r}|\nabla\phi|^2dV,
\]
from which the assertion immediately follows.
\end{proof}

\begin{example}
Given $\alpha>0$, take $\eta$ such that
\[
\eta(t)=\begin{cases}
(-t)^{-\alpha},\ \ \ &t<-1,\\
2t^\alpha,\ \ \  &t>1.
\end{cases}
\]
Then
\begin{itemize}
\item[$(1)$]
$\Lambda_*\asymp\nu_*^2$, so that $M$ is hyperbolic when $\alpha\gg1$.

\item[$(2)$]
$M$ is parabolic if and only if\/ $0<\alpha\leq2$.
\end{itemize}
\end{example}

\begin{proof}
(1) By Proposition \ref{prop:example}, we have
\[
\Lambda_*=\liminf_{r\rightarrow +\infty} \{ r^2 \lambda_1(B_r)\}\geq\frac{\alpha^2}{4}.
\]
On the other hand, since
\[
|B_r|=2\pi\int^r_{-r}\eta'(t)dt=2\pi(\eta(r)-\eta(-r))=4\pi{r}^\alpha-2\pi{r}^{-\alpha},\ \ \ \forall\,r\gg1,
\]
we see that
\[
\nu_*=\liminf_{r\rightarrow +\infty} \frac{\log |B_r|}{\log r}=\alpha.
\]
Thus $\Lambda_*\geq\nu_*^2/4$. This together with Theorem \ref{th:Upper} give $\Lambda_*\asymp\nu_*^2$. In particular, $M$ is hyperbolic provided $\alpha\gg1$, in view of Theorem \ref{th:main}. 

(2) We first verify the \textit{if} part. It suffices to verify that $\mathrm{cap}(B_1)=0$, in view of Theorem \ref{th:Capacity}. Let $\chi:(0,+\infty)\rightarrow[0,1]$ be the Lipschitz continuous function with $\chi=1$ on $[0,1]$, $\chi=0$ on $[r,+\infty)$ and
\[
\chi(t)=\frac{\log{r}-\log{t}}{\log{r}},\ \ \ t\in(1,r).
\]
Let $\psi(t,\theta):=\chi(t)$.Then $\psi$ is a Lipschitz continuous function on $M$ which satisfies $\psi|_{B_1}=1$, $\mathrm{supp}\,\psi\subset{B_r}$ and $|\nabla\psi|=(\log{r})^{-1}t^{-1}$ on $B_r\setminus\overline{B}_1$. Since $\eta'(t)=2\alpha{t^{\alpha-1}}\leq2t$ when $t\geq1$ and $0<\alpha\leq2$, we have
\[
\int_M|\nabla\psi|^2=\frac{2\pi}{(\log{r})^2}\int^r_1\frac{\eta'(t)}{t^2}dt\leq\frac{4\pi}{(\log{r})^2}\int^r_1\frac{dt}{t}=\frac{4\pi}{\log{r}},
\]
i.e., $\mathrm{cap}(B_1)=0$ when $0<\alpha\leq2$.

For the \textit{only if} part, a straightforward computation shows
\[
\Delta{f}=\frac{\partial^2f}{\partial{t}^2}+\frac{\eta''(t)}{\eta'(t)}\frac{\partial{f}}{\partial{t}}+\frac{1}{\eta'(t)^2}\frac{\partial^2f}{\partial\theta^2}-\frac{\eta''(t)}{\eta'(t)^3}\frac{\partial{f}}{\partial\theta}.
\]
In particular, if $f$ is independent of $\theta$, then
\[
\Delta{f}=f''(t)+\frac{\eta''(t)}{\eta'(t)}f'(t)=\frac{(\eta'f')'(t)}{\eta'(t)}.
\]
Note that $\int^{+\infty}_1\frac{ds}{\eta'(s)}<+\infty$ if $\alpha>2$. Let $0<c<\int^{+\infty}_1\frac{ds}{\eta'(s)}$ and $\tau$ a smooth, convex and increasing function on $(-\infty,+\infty)$ such that $\tau(x)\equiv{c}$ for $x\leq{c/2}$ and $\tau(x)=x$ for $x\geq{2c}$. Thus
\[
f(t)=\begin{cases}
\tau\left(\int^t_1\frac{ds}{\eta'(s)}\right),\ \ \ &t\geq1,\\
c,\ \ \ &t\leq1
\end{cases}
\]
gives a nonconstant smooth bounded subharmonic function on $M$, so that $M$ is hyperbolic.
\end{proof}

\begin{example}
Given $\alpha>0$, take $\eta$ such that
\begin{equation}\label{eq:eta_exp}
\eta'(t)=e^{-\alpha|t|},\ \ \ |t|>1.
\end{equation}
Then
\begin{equation}\label{eq:exp_example_2}
\lambda_1(B_r)\gtrsim{e^{-\alpha{r}}},
\end{equation}
and
\[
\liminf_{r\rightarrow+\infty}\frac{-\log(\lambda_1(B_r))}{r}=\alpha=\liminf_{r\rightarrow+\infty}\frac{-\log(|M\setminus B_r|)}{r},
\]
i.e., the estimate in Theorem \ref{th:volume_exponential} is sharp.
\end{example}

\begin{proof}
First of all, since
\[
|M\setminus B_r|=4\pi\int^\infty_re^{-\alpha{t}}dt\asymp{e^{-\alpha{r}}},\ \ \ r\gg1,
\]
we have $\liminf_{r\rightarrow+\infty}\frac{-\log|M\setminus B_r|}{r}=\alpha$, which implies
\[
\liminf_{r\rightarrow+\infty}\frac{-\log(\lambda_1(B_r))}{r}\geq\alpha,
\]
in view of Theorem \ref{th:volume_exponential}.

Next, we shall use the following Hardy-type inequality (cf.  Opic-Kufner \cite{OpicKufner}, pp. 100--103)
\begin{equation}\label{eq:Hardy_exp}
\int^r_{-r}\phi(t)^2\eta'(t)dt\lesssim{e^{\alpha{r}}}\int^r_{-r}\phi'(t)^2\eta'(t)dt,\ \ \ \forall\,\phi\in{C^\infty_0((-r,r))},
\end{equation}
where the implicit constant is independent of $r$. For reader's convenience, we include here a rather short proof for this special case. Since $\int^{+\infty}_{-\infty}\eta'(t)dt$ is finite in view of \eqref{eq:eta_exp}, we have
\begin{equation}\label{eq:Hardy_exp_LHS}
\int^r_{-r}\phi(t)^2\eta'(t)dt\leq\sup_{-r<{t}<{r}}\phi(t)^2\int^r_{-r}\eta'(t)dt\lesssim\sup_{-r<{t}<{r}}\phi(t)^2.
\end{equation}
On the other hand, by setting $|\phi(t_0)|=\sup_{-r<{t}<{r}}|\phi(t)|$, we have
\[
\int^r_{-r}|\phi'(t)|dt\geq\int^{t_0}_{-r}|\phi'(t)|dt\geq\left|\int^{t_0}_{-r}\phi'(t)dt\right|=|\phi(t_0)|=\sup_{-r<{t}<{r}}|\phi(t)|.
\]
This together with Cauchy-Schwarz inequality yield
\begin{equation}\label{eq:Hardy_exp_RHS}
\sup_{-r<{t}<{r}}\phi(t)^2\leq\left(\int^r_{-r}\phi'(t)^2\eta'(t)dt\right)\left(\int^r_{-r}\frac{1}{\eta'(t)}dt\right)\lesssim{e^{\alpha{r}}}\int^r_{-r}\phi'(t)^2\eta'(t)dt.
\end{equation}
By \eqref{eq:Hardy_exp_LHS} and \eqref{eq:Hardy_exp_RHS}, we obtain \eqref{eq:Hardy_exp}, which in turn gives \eqref{eq:exp_example_2}, i.e.,
\[
\liminf_{r\rightarrow+\infty}\frac{-\log(\lambda_1(B_r))}{r}\leq\alpha.
\]
\end{proof}

\begin{remark}
By Proposition \ref{prop:example}, we only obtain a weaker conclusion
\[
\lambda_1(B_r)\geq\frac{1}{4}\frac{\eta'(r)^2}{\eta(r)^2}\gtrsim{e^{-2\alpha{r}}}.
\]
\end{remark}

\begin{example}

Let $\mu$ be a positive, smooth and decreasing function on $[1,+\infty)$ satisfying
\begin{itemize}
\item[$(1)$]
$\lim_{t\rightarrow+\infty}\mu(t)=0$,

\item[$(2)$]
$\int^{+\infty}_1\mu(s)ds=+\infty$,

\item[$(3)$]
$t\mu(t)$ is increasing on $[c,+\infty)$ for some $c\gg1$.
\end{itemize}
Take $\eta$ such that
\[
\eta(t)=\begin{cases}
e^{-\int^{-t}_1\mu(s)ds},\ \ \ &t<-1,\\
2e^{\int^t_1\mu(s)ds},\ \ \ &t>1.
\end{cases}
\]
Then
\begin{equation}\label{eq:lambda_mu}
\lambda_1(B_r)\asymp\mu(r)^2.
\end{equation}
\end{example}

\begin{proof}
Note that $\eta'(t)/\eta(t)=\mu(-t)$ for $t<-1$ and $\eta'(t)/\eta(t)=\mu(t)$ for $t>1$. Thus it follows from Proposition \ref{prop:example} that
\begin{equation}\label{eq:mu_behavior_lower}
\lambda_1(B_r)\geq\frac{1}{4}\inf_{|t|\leq{r}}\frac{\eta'(t)^2}{\eta(t)^2}=\frac{\mu(r)^2}{4}.
\end{equation}

On the other hand, we have $r\mu(r)\geq{c\mu(c)}>0$ for $r\geq{c}\gg1$ in view of the condition (3). Thus we may take $0<\varepsilon\leq{c\mu(c)}/2$ so that
\begin{equation}\label{eq:r_epsilon}
r_\varepsilon:=r-\varepsilon\mu(r)^{-1}=r\left(1-\varepsilon{r}^{-1}\mu(r)^{-1}\right)\geq\frac{r}{2},\ \ \ \forall\,r\geq{c}.
\end{equation}
Set $I_r:=(-r,-r_\varepsilon)$. Since $\eta''(t)=-\mu'(-t)\eta(t)+\mu(-t)\eta'(t)\geq0$, i.e., $\eta'(t)$ is increasing, on $(-\infty,-1]$, it follows that
\begin{eqnarray*}
\lambda_1(B_r)
&\leq& \lambda_1\left(\{(t,\theta)\in{M}:-r\leq{t}\leq{-r_\varepsilon}\}\right)\\
&\leq& \inf_{\phi\in{C^\infty_0(I_r)}}\left\{\frac{\int_{I_r}\phi'(t)^2\eta'(t)dt}{\int_{I_r}\phi(t)^2\eta'(t)dt}\right\}\\
&\leq& \inf_{\phi\in{C^\infty_0(I_r)}}\left\{\frac{\int_{I_r}\phi'(t)^2dt}{\int_{I_r}\phi(t)^2dt}\right\}\cdot\frac{\eta'(-r_\varepsilon)}{\eta'(-r)}\\
&=& \lambda_1(I_r)\cdot\frac{\eta'(-r_\varepsilon)}{\eta'(-r)}.
\end{eqnarray*}
Since $\lambda_1(I_r)\lesssim|I_r|^{-2}\asymp\mu(r)^2$, we obtain
\begin{equation}\label{eq:mu_behavior_upper}
\lambda_1(B_r)\lesssim\mu(r)^2\cdot\frac{\eta'(-r_\varepsilon)}{\eta'(-r)}.
\end{equation}
We have
\[
\frac{\eta'(-r_\varepsilon)}{\eta'(-r)}=\frac{\mu(r_\varepsilon)}{\mu(r)}\exp\left(\int^r_{r_\varepsilon}\mu(s)ds\right)\leq\frac{\mu(r_\varepsilon)}{\mu(r)}\exp\left(\varepsilon\frac{\mu(r_\varepsilon)}{\mu(r)}\right),
\]
for $\mu$ is decreasing and $r-r_\varepsilon=\varepsilon\mu(r)^{-1}$. By condition (3) and \eqref{eq:r_epsilon}, we have
\[
\frac{\mu(r_\varepsilon)}{\mu(r)}\leq\frac{r}{r_\varepsilon}\leq2.
\]
Thus $\frac{\eta'(-r_\varepsilon)}{\eta'(-r)}=O(1)$ as $r\rightarrow+\infty$. This together with \eqref{eq:mu_behavior_lower} and \eqref{eq:mu_behavior_upper} give \eqref{eq:lambda_mu}.
\end{proof}

Particular choices of $\mu$ give the following
\begin{itemize}
\item[$(1)$]
For $\mu(t)=t^{-1}(\log{t})^\beta$ with $\beta>0$, $\lambda_1(B_r)\asymp{r^{-2}}(\log{r})^{2\beta}$.

\item[$(2)$]
For $\mu(t)=t^{-\alpha}$ with $0<\alpha<1$, $\lambda_1(B_r)\asymp{r^{-2\alpha}}$.

\item[$(3)$] For $\mu(t)=(\log(t+1))^{-\gamma}$ with $\gamma>0$, $\lambda_1(B_r)\asymp(\log{r})^{-2\gamma}$.
\end{itemize}
\noindent In all three cases, we have
\[
\Lambda_*=\liminf_{r\rightarrow+\infty}\left\{r^2\lambda_1(B_r)\right\}=+\infty.
\]
Thus these Riemannian manifolds $(M,g)$ are hyperbolic in view of Theorem \ref{th:main}.

\end{document}